\documentclass[12pt]{amsart}
\usepackage{amsmath,amsthm,amsfonts,amssymb,latexsym,amsopn,amscd}
\usepackage{color}

\theoremstyle{plain}
\newtheorem{teo}{Theorem}[section]
\newtheorem{lem}[teo]{Lemma}

\newtheorem{cor}[teo]{Corollary}

\newtheorem*{teo*}{Theorem}
\newtheorem*{cor*}{Corollary}

\newtheorem*{definition*}{Definition}
\newtheorem*{hypA}{Hypotheses A}

\theoremstyle{definition}

\theoremstyle{remark}
\newtheorem{rem}[teo]{Remark}

\DeclareMathOperator{\GL}{GL}
\DeclareMathOperator{\SL}{SL}
\DeclareMathOperator{\PSL}{PSL}
\DeclareMathOperator{\OO}{O}

\DeclareMathOperator{\Z}{Z}		
\DeclareMathOperator{\C}{C}

\DeclareMathOperator{\PSp}{PSp}		      		
\DeclareMathOperator{\Sp}{Sp}		      		
\DeclareMathOperator{\PSU}{PSU}		      		
\DeclareMathOperator{\SU}{SU}

\DeclareMathOperator{\Pom}{P\Omega}

\newcommand{\FF}{\mathbf{F}}
\newcommand{\ZZ}{\mathbf{Z}}

\newcommand{\CC}{\mathbf{C}}

\title[Jordan's Theorem]{Boosting an analogue of Jordan's theorem for finite groups}

\author[I. Mundet]{Ignasi Mundet i Riera}
\address{Departament d'\`Algebra i Geometria\\
Facultat de Matem\`atiques\\
Universitat de Barcelona\\
Gran Via de les Corts Catalanes 585\\
08007 Barcelona \\
Spain}
\email{ignasi.mundet@ub.edu}

\author[A. Turull]{Alexandre Turull}
\address{Department of Mathematics \\
University of Florida \\
Gainesville, FL 32611 \\
USA}
\email{turull@ufl.edu}

\date{}

\begin{document}

\begin{abstract}
Let $\mathcal C$ be a set of finite groups which is closed under taking subgroups
and let $d$ and $M$ be positive integers.
Suppose that for any $G\in\mathcal C$
whose order is divisible by at most two distinct primes there exists an abelian
subgroup $A\subseteq G$ such that $A$ is generated by at most $d$ elements
and $[G : A] \le M$.
We prove that there exists a positive constant $C_0$
such that any $G \in \mathcal C$ has an abelian subgroup
$A$ satisfying $[G : A] \le C_0$, and $A$ can be
generated by at most $d$ elements. We also prove some related results.
Our proofs use the Classification of Finite Simple
Groups.
\end{abstract}


\maketitle

\section{Introduction}

A celebrated theorem of C. Jordan \cite{jordan} states the following.

\begin{teo*}[C. Jordan]
For any $d$ there exists a constant $C$ such that any finite subgroup
$G$ of $\GL(d,\CC)$ has an abelian subgroup
$A\subseteq G$ of index  at most $C$, and $A$ can be
generated by at most $d$ elements.
\end{teo*}

The claim on the number of generators of $A$ is not usually
part of the statement of Jordan's theorem, but from the perspective of this paper
it is a natural complement to the usual statement.
To prove it, note that if $A$ is a finite abelian subgroup of $\GL(d,\CC)$
then, by simultaneous diagonalization,
$A$ is isomorphic to a subgroup of the group $D$ of
$d\times d$ diagonal matrices whose diagonal entries
are $|A|$-th roots unity; the group $D$ is abelian and
can be generated by at most $d$ elements, so any subgroup of $D$ can be
generated by at most $d$ elements.

In an alternative version of Jordan's theorem, it is
claimed that finite subgroups of $\GL(d,\CC)$ have a
{\it normal} abelian subgroup of index bounded above by a constant
depending only on $d$.
However, since $[G : A] = n$ implies $[G : \cap_{g \in G} \,g^{-1}Ag] \le n!$,
this version
and the one we stated are equivalent, up to using a different bound $C$.
We note that, using the Classification of Finite Simple Groups, M. Collins \cite{collins}
gives, for all $d$, the best possible bound for this alternative
version.

Our aim in this paper is to provide a tool to prove the conclusion of Jordan's
Theorem for certain classes of finite groups.
To put our result in context, we first review some extensions
and analogues of Jordan's theorem.

Jordan's theorem can be generalized taking instead of $\GL(d,\CC)$
any (finite dimensional) Lie group $R$ with a finite number
of connected components: for any such $R$
there exist constants $C(R)$ and $d(R)$ with the property that
any finite subgroup of $R$ has an abelian subgroup of index at most $C(R)$
which can be generated by at most $d(R)$ elements. This follows from
the existence and uniqueness up to conjugation of maximal compact subgroups
(see \cite[Theorem 14.1.3]{HN}),
Peter--Weyl's theorem (see \cite[Chap. III, \S 4, Theorem 4.1]{BD}),
and the above version of Jordan's theorem (see also \cite{BW}).

Similar to Jordan's theorem, a result of Brauer and Feit \cite{BF}
states that if $K$ is a field of characteristic $p>0$ then
any finite subgroup $G$ of $\GL(d,K)$ has an abelian subgroup whose index
is bounded above by a constant depending only on $d$ and on the size of the $p$-Sylow
subgroups of $G$.
Work of M. Collins \cite{collins2} provides an interesting different
modular analogue of Jordan's Theorem.

Remarkably, the conclusion of Jordan's theorem is also known
or expected to be true when $\GL(d,\CC)$ is replaced by some particular
{\it much bigger} groups.
Two notable examples are Serre's conjecture on the Cremona group
(see \cite[\S 6.1]{S2}; some partial results appear in
\cite[Theorem 5.3]{S2}, \cite{PS,S}; see also \cite{Z}), and
Ghys's conjecture on the diffeomorphism group of smooth compact manifolds
(see \cite[Question 13.1]{fisher}\footnote{The statement of Ghys's conjecture as written in \cite{fisher}
does not mention any bound on the number of generators of the abelian subgroup;
the existence of such a bound (once the abelian subgroup has been
proved to exist) follows from \cite{MS}.}
and, for some partial results,
\cite[\S 5]{Zi} and \cite{M1}). In both cases, the {\it big group} on
whose finite subgroups one is interested can be suitably understood as an
infinite dimensional Lie group (see \cite{Dem}
for the case Cremona groups and \cite{Mil} for the case of diffeomorphism groups).
See \cite{popov} for a nice survey on these questions, mostly centered on Serre's conjecture
and some natural extensions of it.

Most of these statements and conjectures share the following pattern:
one is given a set
of finite groups $\mathcal C$, closed under taking subgroups,
and one proves (or wants to prove) that there exist constants $C_0$ and $d$ such that
any $G\in\mathcal C$ has an abelian subgroup of index at most $C_0$ which can be
generated by at most $d$ elements.
It will be useful for us to encode this property in a definition.

\begin{definition*}
Let $\mathcal C$ be a set
of finite groups, and let $C_0$ and $d$ be
positive integers. We say that $\mathcal C$ \emph{satisfies the Jordan property}
$\mathcal J(C_0,d)$ if for every element $G \in \mathcal C$ there exists some
abelian subgroup $A$ of $G$ such that $[G : A] \le C_0$ and $A$ can be
generated by at most $d$ elements.
\end{definition*}

A similar notion has been introduced and studied by
V.L. Popov  \cite{popov1,popov}, according to which a group $G$ is a Jordan
group\footnote{Note that in the literature on permutation groups there
exists a concept called {\it Jordan group}, which was introduced by
W.M. Kantor \cite{kantor} (see also \cite[Chap. 6, \S 6.8]{cameron}) and which is different
from the concept of Jordan group
introduced by Popov. The only thing in common between the two notions seems to be
that they are both inspired on theorems of C. Jordan --- but not on the same one!
(see \cite[Chap. 6, Theorem 6.15]{cameron} for the theorem of Jordan
on which Kantor's terminology is based).}
if there exists a constant $C$ such that any finite subgroup of $G$
has an abelian subgroup of index at most $C$.
In \cite{popov1} a group $G$ is said to have Jordan property if it is a Jordan group
in this sense
(this terminology is also used by Y. Prokhorov and C. Shramov in \cite{PS}).

More generally, we say that a set
of finite groups $\mathcal C$
satisfies the Jordan property if there exist numbers $C_0$ and $d$ such that
$\mathcal C$ satisfies $\mathcal J(C_0,d)$.
The main result of this paper implies that
if a set
of finite groups $\mathcal C$ is closed under taking subgroups then, in order to check whether
$\mathcal C$ satisfies the Jordan property, it suffices to
consider the subset of $\mathcal C$ consisting of groups whose
cardinal has at most two different prime divisors: if this subset has
the Jordan property, then so does $\mathcal C$, although possibly with
different constants. (Our main result is in fact slightly stronger than
this.)

In order to give a precise statement, we introduce the following
notation. If $\mathcal C$ is a set
of finite groups, we denote by
$$\mathcal T(\mathcal C)\subset{\mathcal C}$$
the set of all $T \in \mathcal C$ such that
there exist primes $p$ and $q$, a Sylow $p$-subgroup $P$ of $T$,
and a normal Sylow $q$-subgroup $Q$ of $T$, such that $T = PQ$.
(In particular, $T\in{\mathcal T}({\mathcal C})$ implies
$|T|=p^{\alpha}q^{\beta}$
for some primes $p$ and $q$ and some nonnegative integers $\alpha,\beta$).

Our main result is the following theorem (Theorem \ref{te:main}).

\begin{teo*}
Let $d$ and $M$ be positive integers.
Let $\mathcal C$ be a set of finite groups which is closed under taking subgroups
and such that $\mathcal T(\mathcal C)$ satisfies the Jordan property $\mathcal J(M,d)$.
Then there exists a positive integer $C_0$ such that
$\mathcal C$ satisfies the Jordan property $\mathcal J(C_0,d)$.
\end{teo*}

This theorem is used in \cite{mundet} to prove  Ghys's conjecture
for manifolds without odd cohomology.

Our proof makes use of the Classification of Finite Simple Groups. While
our proof would provide a specific value for the constant $C_0$ in terms of
$M$ and $d$, we make
no attempt to make it explicit or to find the best possible value.

The following corollary implies that the number $C_0$ in our main theorem
can be chosen to depend only on $M$ and $d$, but not on the particular set
${\mathcal C}$. We use the following notation. If $G$ is a finite group, then
${\mathcal S}(G)$ denotes the set of all subgroups of $G$, and we define
${\mathcal T}(G)={\mathcal T}({\mathcal S}(G))$. We also denote by
${\mathcal A}_d(G)$ the set of abelian subgroups of $G$ which can be
generated by at most $d$ elements.

\begin{cor}
\label{cor:C-0-indep}
Given positive integers $d$ and $M$ there exists an integer $C_0$ with the
following property. Let $G$ be any finite group.
Suppose that for any $T\in{\mathcal T}(G)$ there exists some
$A\in{\mathcal A}_d(T)$ satisfying $[T:A]\leq M$.
Then there exists some $B\in{\mathcal A}_d(G)$ satisfying $[G:B]\leq C_0$.
\end{cor}
\begin{proof}
Suppose the corollary is false.  Then, there exist some integers $d$ and $M$
such that, for them, there
does not exist any integer $C_0$ with the property specified in the statement.
This implies that there exists a sequence of finite groups $G_1,G_2,G_3,\dots$ such that
for any $i$ and any $T\in {\mathcal T}(G_i)$ there exists some
$A\in {\mathcal A}_d(T)$ satisfying $[T:A]\leq M$ and,
if we define
$$C(G_i):=\inf \{[G_i:B] \mid B\in{\mathcal A}_d(G_i)\},$$
then $C(G_i)\to\infty$.
Let ${\mathcal C}=\bigcup_i{\mathcal S}(G_i)$. Then ${\mathcal T}({\mathcal C})$
satisfies ${\mathcal J}(M,d)$ but ${\mathcal C}$ does not satisfy
${\mathcal J}(C_0,d)$ for any value of $C_0$. This  contradicts our main
theorem, so the corollary is proved.
\end{proof}

One can obtain variations on the theme of our main theorem and the
previous corollary
replacing $\mathcal T(\mathcal C)$ by any subset of $\mathcal C$
containing $\mathcal T(\mathcal C)$.
Two natural choices are the following ones:
$$
\mathcal T_0(\mathcal C) = \{ T_0 \in \mathcal C: \text{ $T_0$ is a $\{p,q\}$-group for some primes $p$ and $q$}\},
$$
$$
\mathcal Sol(\mathcal C) = \{ S \in \mathcal C: \text{ $S$ is solvable}\}.
$$
We have inclusions $\mathcal T(\mathcal C)\subset \mathcal T_0(\mathcal C)$
and $\mathcal T(\mathcal C)\subset \mathcal Sol(\mathcal C)$:
the first one is obvious and the second one is an easy exercise
(in fact, Burnside's $p^\alpha q^\beta$-theorem, see e.g. \cite[Theorem 7.8]{isa},
implies that we also have an inclusion
$\mathcal T_0(\mathcal C)\subset \mathcal Sol(\mathcal C)$).
Hence, combining our main theorem and Corollary \ref{cor:C-0-indep}
we obtain immediately the following.

\begin{cor}
Given positive integers $d$ and $M$ there exists an integer $C_0$
with the following property.
Let $\mathcal C$ be a set of finite groups which is closed under taking subgroups.
If $\mathcal Sol(\mathcal C)$ satisfies the
Jordan property $\mathcal J(M,d)$,
then  $\mathcal C$ satisfies the Jordan property $\mathcal J(C_0,d)$.
\end{cor}

The same corollary holds true replacing
$\mathcal Sol(\mathcal C)$ by $\mathcal T_0(\mathcal C)$.

A natural question is whether one could strengthen our main theorem
replacing $\mathcal T(\mathcal C)$ by some (in general) smaller subset. Although we can
not answer completely
this question at present, we can at least prove
that it is not possible to replace $\mathcal T(\mathcal C)$
by
$$
\mathcal P(\mathcal C) = \{ P \in \mathcal C: \text{ $P$ is a $p$-group for some prime $p$}\}.
$$
To justify this claim, let us denote by $G_p$ the group
of affine transformations of the affine line over the finite field
$\FF_p$, where $p$ is any prime. (In particular, $|G_p| = p(p-1)$.)
Let
$$
\mathcal L =\{G\mid G\text{ is a subgroup of }G_p\text{ for some prime $p$}\}.
$$
The set $\mathcal L$ does not satisfy the Jordan property
$\mathcal J(C,d)$ for any $C$ and any $d$
(this follows from Lemma \ref{le:frobeniusgroups} below),
and yet all elements of $\mathcal P(\mathcal L)$ are abelian
and cyclic. Indeed, for each $p$ we have an exact sequence
$0\to\FF_p\to G_p \to\FF_p^*\to 1$, where
$\FF_p^*\subset\FF_p$ is the multiplicative group of units
and $\FF_p$ is the additive group; both $\FF_p^*$ and $\FF_p$ are cyclic
and their orders are coprime, so all Sylow subgroups
of $G_p$ are abelian and cyclic.
Hence, $\mathcal P(\mathcal L)$ satisfies the Jordan property
$\mathcal J(1,1)$, but $\mathcal L$ does not satisfy the Jordan property
$\mathcal J(C,d)$ for any $C$ and any $d$.

We close this introduction with a remark on style.
This is a paper on finite groups which, we hope, will also be of interest to
mathematicians whose main expertise is outside finite group theory.
With these readers in mind and to make the paper generally
easier to read, we give detailed
and complete references to more of the results we use than we normally
would if we were writing for an audience of only finite group theorists.
An excellent reference for the basic notions and results on
finite groups which we use is \cite{isa}.

After completing this paper, we were informed by L\'aszl\'o Pyber that
he had independently obtained some results related to the ones in this paper,
but that his results have not yet appeared in print.

\section{Preliminary lemmas} \label{sec:preliminary}

Recall that a \emph{quasisimple} group is a finite group $S$ such that
$S/\Z(S)$ is a non-abelian simple group and $S$ is perfect.
Here a finite group $S$ is perfect if $S' = S$, that is,
$S$ has no nontrivial abelian homomorphic images.
Following standard conventions as, for example, in \cite{isa}, we define
the \emph{layer} $E(G)$ of a finite group
$G$ to be  the product of all the quasisimple subnormal
subgroups of $G$  (the latter are called the \emph{components} of $G$).
For two subgroups $A$ and $B$ of a finite group $G$, we denote
by $[A,B]$ the \emph{commutator subgroup} of $A$ and $B$. We recall
\cite[Lemma 4.3]{isa} that $A$ normalizes $B$ if and only if
$[A,B] \subseteq B$.

\begin{lem} \label{le:nonnormalpsubgroup}
Let $G$ be a finite group,
let $G_1 = \C_G(E(G))$, let $p$ be a prime, and let $P$ be an
abelian Sylow $p$-subgroup
$G_1$.  Assume that $P$ is not a normal subgroup of $G_1$.
Then there exists a prime $q$ and a nontrivial $q$-subgroup $Q$ of $G_1$
such that $p \ne q$ and $[P,Q] = Q$.
(In particular, $P$ normalizes $Q$.)
\end{lem}

\begin{proof}
Notice that $F(G)$, the Fitting subgroup of $G$, commutes
with $E(G)$ (for example \cite[Theorem 9.7]{isa}) so that
$F(G) \subseteq G_1$, and more precisely $F(G) \subseteq F(G_1)$.
Since $G_1$ is a normal subgroup of
$G$ and $F(G_1)\subseteq G_1$ is characteristic,
we also know that $F(G_1) \subseteq F(G)$, and therefore
we have $F(G) = F(G_1)$. Denote as usual by $\OO_p(G)$ the largest
normal $p$-subgroup of $G$. Notice that $\OO_p(G)$ is the
Sylow $p$-subgroup of $F(G)$, and it is a normal subgroup of $G_1$.
Since $P$ is a Sylow $p$-subgroup of $G_1$, we have $\OO_p(G) \subseteq P$.
In particular, since $P$ is abelian, $P$ centralizes $\OO_p(G) = \OO_p(G_1)$.
Suppose that $P \subseteq F(G)$. Then $P$ is a Sylow $p$-subgroup
of $F(G)$, and this implies that $P = \OO_p(G)$, contradicting
the fact that $P$ is not a normal subgroup of $G_1$.
It follows that
$P \not\subseteq F(G) = F(G_1)$.
We claim that $E(G_1)$ is trivial. Otherwise, $G_1$ contains at least one component.
Any quasisimple subnormal subgroup of $G_1$ is contained in $E(G)$.
Since $G_1$ is also in the centralizer of $E(G)$ in $G$, this implies
that any such subnormal subgroup is in $\Z(E(G))$, which implies
that it is solvable, a contradiction.
Therefore,
$E(G_1) = 1$ as claimed. In particular, the generalized Fitting subgroup
of $G_1$ is $F(G)$.
By for example \cite[Corollary 9.9]{isa}, this implies that there exists some
prime $q$ such that $[P,\OO_q(G)] \ne 1$.
By the remark above, $p \ne q$. Set $Q = [P,\OO_q(G)]$.
It follows by, for example, \cite[Lemma 4.29]{isa} that $Q = [P, Q]$.
The lemma follows.
\end{proof}

\begin{lem} \label{le:propersubgroup}
Let $H = PQ$ be a finite group which is the product of
a $p$-subgroup $P$ and a normal subgroup $Q$,
where $p$ is a prime, and $[P,Q] = Q$.
Let $J$ be any proper subgroup of $H$.
Then $[H : J] \ge p$.
\end{lem}

\begin{proof}
Let $\Omega$ be the set of left cosets of $J$ in $H$.
Then $\left|\Omega\right| > 1$, and the action by left
multiplication provides a group homomorphism
$$
\rho : H \to S_{\Omega}
$$
from $H$ to the symmetric group on $\Omega$. Since the
image of $\rho$ is transitive, this image is not trivial.
Now suppose $\left|\Omega\right| < p$. Then $\rho(P)$ is
trivial since the Sylow $p$-subgroup of $S_{\Omega}$ is
trivial, and it follows from $H=P[P,Q]$ that $\rho(H)$ is trivial. This
is a contradiction. Hence, the lemma holds.
\end{proof}

Recall that the \emph{Frattini}
subgroup $\Phi(G)$ of a finite group $G$ is the intersection of all the maximal
subgroups\footnote{By convention, \emph{maximal subgroup} means a
subgroup which is maximal among the \emph{proper} subgroups.} of $G$.

\begin{lem} \label{le:centralizeroffrattiniquotient}
Let $G$ be a finite group. Suppose that $E(G) = 1$.
Then $F(G)/\Phi(F(G))$ is an abelian group of square-free exponent
and
$$
\C_G(F(G)/\Phi(F(G))) = F(G).
$$
\end{lem}

\begin{proof}
$F(G)$ is the direct product of all the $\OO_p(G)$ for all primes $p$.
It follows that $\Phi(F(G))$ is the direct product of all the
$\Phi(\OO_p(G))$ for all primes $p$. Therefore
$F(G)/\Phi(F(G))$ is isomorphic to the direct product of
the $\OO_p(G)/\Phi(\OO_p(G))$ for all primes $p$. By, for example, \cite[1D.8]{isa},
we have that $\OO_p(G)/\Phi(\OO_p(G))$ is elementary abelian,
and it follows that $F(G)/\Phi(F(G))$ is abelian of square-free exponent.

Set
$$
C_1 = \C_G(F(G)/\Phi(F(G))).
$$
It is clear that $C_1$ is a normal subgroup of $G$ such that $F(G) \subseteq C_1$.
Let us assume that $F(G) \ne C_1$.
Then the set $\mathcal S$ of subnormal subgroups of $G$ contained in $C_1$ and
not contained in $F(G)$ is nonempty, because $C_1\in\mathcal S$.
Let $C_2$ be minimal among the elements of $\mathcal S$.

Suppose that $C_2F(G)/F(G)$ is not abelian. Then $C_2F(G)/F(G)$ is a non-abelian
simple group. Let $\Gamma=C_2F(G)/F(G)$. We have inclusions
$\Gamma'\subseteq C_2'F(G)/F(G)\subseteq C_2F(G)/F(G)=\Gamma$
and, since $\Gamma'=\Gamma$, we have
$C_2'F(G)/F(G)=C_2F(G)/F(G)$. This implies that
$C_2'\not\subseteq F(G)$, because $C_2\not\subseteq F(G)$. Hence, $C_2'\in{\mathcal S}$, so,
by the minimality of $C_2$, we have that $C_2= C_2'$ and $C_2$ is perfect.
Let $p$ be any prime. Then $C_2$ acts trivially
on $\OO_p(G)/\Phi(\OO_p(G))$. Let $q$ be a prime divisor of $|C_2F(G)/F(G)|$ with $p \ne q$,
and let $Q$ be a Sylow $q$-subgroup of $C_2$. Then by, for example, \cite[3D.4]{isa},
we have that, since $Q$ centralizes $\OO_p(G)/\Phi(\OO_p(G))$, we also have that
$Q$ centralizes $\OO_p(G)$. This implies that $\C_{C_2}(\OO_p(G))$ is a subnormal
subgroup of $G$ contained in $C_2$, and, since $Q \subseteq \C_{C_2}(\OO_p(G))$,
$\C_{C_2}(\OO_p(G))$ is not contained in $F(G)$. Again by the minimality, we obtain
that $C_2= \C_{C_2}(\OO_p(G))$. Since this happens for
all primes $p$, it follows that $C_2$ centralizes $F(G)$.
Since $E(G)=1$, by \cite[Theorem 9.8]{isa} we deduce that $C_2\subseteq F(G)$,
which contradicts the definition of $C_2$.

Hence, $C_2F(G)/F(G)$ is abelian. By the minimality of $C_2$, we
have that $C_2 F(G)/F(G)$ has prime order, and
we set $p=|C_2 F(G)/F(G)|$. Let $P$ be a Sylow $p$-subgroup of $C_2$. Then,
by the argument above, for every prime $q$ with $p \ne q$, we have
that $P$ centralizes $\OO_q(G)$.  We also have that $P \OO_p(G)$
centralizes $\OO_q(G)$. Since $P \OO_p(G)$ is a Sylow $p$-subgroup of
$C_2 F(G)$, it follows that $P \OO_p(G)$ is normalized by a
Sylow $r$-subgroup of $C_2 F(G)$ for every prime $r$, so that
$P \OO_p(G)$ is a normal $p$-subgroup of
$PF(G)$. This implies that $P \OO_p(G)$ is a subnormal $p$-subgroup of
$G$ and this implies, by for example \cite[Theorem 2.2]{isa},
that $P \subseteq F(G)$. This final contradiction completes the
proof of the lemma.
\end{proof}

The following lemmas will be used later (in Lemma \ref{le:quasisimple})
to prove that if a set
of groups ${\mathcal C}$ satisfies the hypothesis of our main theorem,
then the isomorphism classes of groups appearing as components of elements
of  ${\mathcal C}$ form a finite collection. The crucial ingredient is to control
the possible non-abelian simple groups, and for that we will use
the Classification of the Finite Simple Groups \cite{GLS,W}.

The classification provides an infinite list of finite simple
groups such that every finite simple group is isomorphic to a member of
this list. There are unfortunately discrepancies with the
notation for finite simple groups in the literature. In this paper, we  follow the
notation in \cite{GLS} for the list of finite simple groups.
Table I in p. 8 of [op.cit.] also provides a list of alternative
notations for these finite simple groups. While simple groups on this list
can be isomorphic to other groups on the list in some cases, these few isomorphisms
are all listed in \cite[Table II, p. 10]{GLS}.
The table begins with the familiar abelian simple groups of prime order $\ZZ_p$, and
the alternating groups $A_n$. This is followed by 16 families of groups which all depend
on a parameter $q$ (and some also on another parameter $n$):
these are the \emph{finite simple groups of Lie type}.
The table is then followed by the 26 \emph{sporadic simple groups}. These cover,
up to isomorphism, all the finite simple groups except the \emph{Tits group} $^2F_4(2)'$ which appears
in the table in footnote 2.

A description of the finite simple groups of Lie type appears
in \cite{carter}. As we will see below, those finite simple groups of Lie type which depend
on two parameters $q$ and $n$ are isomorphic to \emph{classical groups}.
Of course, many classical groups are not simple groups. We use the notation
of \cite{W} for the classical groups.

The following result refers to finite groups of Lie type and is probably
well known (see for example \cite{Z}; in fact, there are much stronger results
in the literature, e.g. \cite{LNS}). We include a short proof with
references for completeness.

\begin{lem} \label{le:psl2subgroupsofsimplegroups}
Let $S$ be a finite simple group of Lie type, not isomorphic to a Suzuki group
$\ ^2B_2(2^{2n+1})$.
Let $q$ be the parameter corresponding to $S$ according to \cite[Table I, p. 8]{GLS}.
Assume that $q\geq 4$.
Then $S$ contains a subgroup isomorphic to some (possibly trivial) central
extension of $\PSL(2,q)$.
\end{lem}

\begin{rem} The Suzuki groups $\ ^2B_2(2^{2n+1})$ don't contain subgroups isomorphic
to a central extension of $\PSL(2,2^{2n+1})$ because their order
$|^2B_2(q)|=q^2(q-1)(q^2+1)$ is not divisible by $3$, whereas
the order of $\PSL(2,2^{2n+1})$ is divisible by $3$.
\end{rem}

\begin{proof}
Suppose that $S$ is isomorphic to an untwisted group of Lie type.
By \cite[Theorem 6.3.1]{carter} there exists a homomorphism
from $\SL(2,q)$ to $S$ with nontrivial image (see the formulas for
$x_r(t)$ in \cite[p. 64]{carter}).
Since we assume that $q \ge 4$, the group $\PSL(2,q)$ is
simple, so the image of this homomorphism is isomorphic to either
$\SL(2,q)$ or $\PSL(2,q)$.
For any $n\geq 2$
the simple group $^2A_n(q) \simeq \PSU_{n+1}(q)$
contains a
subgroup isomorphic to a central extension of $\PSU_2(q)$, and by
\cite[\S 2.6.1]{W} there is an isomorphism $\PSU_2(q)\simeq \PSL(2,q)$,
so the result holds for $^2A_n(q)$.
The simple group $^2D_n(q)$ (for $n \ge 2$) contains a subgroup isomorphic to $^2D_2(q)$
\cite[Theorem 14.5.2]{carter} and by \cite[Table II, p. 10]{GLS} $^2D_2(q)$ is
isomorphic to $A_1(q^2)$. Since
$A_1(q^2)\simeq\PSL_2(q^2)$, the result
also holds for $^2D_n(q)$.
We have the following containments:
$^3D_4(q) > G_2(q)$ \cite[\S 4.6.5, Theorem 4.3]{W},
$^2G_2(3^{2m+1}) > \PSL(2,3^{2m+1})$ \cite[\S 4.5.3, Theorem 4.2]{W},
$^2F_4(2^{2m+1}) > \SU_3(2^{2m+1})$ \cite[\S 4.9.3, Theorem 4.5]{W},
and $^2E_6(q) > F_4(q)$ \cite[p. 173]{W}.
Since we have already proved the result for the smaller of these groups
in each of the four cases
(or for its quotient by a central subgroup,
e.g. $^2A_2(2^{2m+1})\simeq \PSU_3(2^{2m+1})$ is such a quotient of
$\SU_3(2^{2m+1})$)
the conclusion of the lemma holds for each of the larger groups.
Hence, the proof of the lemma is now complete.
\end{proof}

Recall that a finite permutation group $G$ on a set $\Omega$ is called a \emph{Frobenius group}
if $G$ is transitive, no element of $G$ except the identity fixes more than one element of $\Omega$,
and some non-identity element of $G$ fixes some element of $\Omega$
\cite[p. 191]{A}. (For example, the group of affine transformations
of the affine line over a finite field is a Frobenius group.) If $\omega \in \Omega$, the
stabilizer $G_{\omega}$ of $\omega$ in $G$ is called a Frobenius \emph{complement}.
Since $G$ acts transitively on $\Omega$, all Frobenius complements are conjugate,
and hence they all have the same order.
More generally, a finite group is called a
\emph{Frobenius group} if it is isomorphic to some permutation group which is a Frobenius group.

\begin{lem} \label{le:frobeniussubgroupsofsimplegroups}
Let $S$ be a finite simple group of Lie type.
Let $q$ be the parameter corresponding to $S$ according to \cite[Table I, p. 8]{GLS}.
Suppose that $q\geq 4$.
Then $S$ contains subgroups $M$ and $N$, such that
$M$ is solvable, $N$ is a normal subgroup of $M$, and $M/N$ is a Frobenius group with
complement of order at least $(q - 1)/2$.
\end{lem}

\begin{proof}
Suppose first that $S$ is not a Suzuki group. Then, by Lemma \ref{le:psl2subgroupsofsimplegroups},
$S$ has a subgroup which is a central extension of $\PSL(2,q)$.
Now $\PSL(2,q)$ acts doubly transitively on the projective line of $\FF_{q}$,
in such a way that no non-identity element fixes more than two points.
Let $B$ be  the stabilizer of a point on the projective line. Then $B$ is a solvable group of order
$q(q -1)$ if $q$ is even, and $q(q -1)/2$ if $q$ is odd.
It follows immediately from the definition of Frobenius group, that $B$ is a
Frobenius group in its action on the affine line. Since its Frobenius
complement has order either $q - 1$ or $(q - 1)/2$ the lemma holds in this
case.

Now suppose that $S$ is a Suzuki group. By Suzuki's original article
\cite{suzuki}, $S$ has order $q^2(q -1)(q^2 + 1)$,
and acts faithfully and doubly transitively on a set
$\Delta$ in such a way that $|\Delta| = q^2 + 1$ and no non-identity element of $S$ fixes more than two elements
of $\Delta$. Set $B$ to be the stabilizer of a point in $\Delta$. Then $B$ is a Frobenius
group with Frobenius complement of order $q -1$.
The group $S$ can be identified with a subgroup of $\GL_4(2^{2n+1})$ \cite[\S 4.2.1]{W},
and with respect to this identification $B$ corresponds to a subgroup consisting
of lower triangular matrices \cite[\S 4.2.2]{W}. Hence $B$ is solvable,
so the lemma holds for Suzuki groups as well.
\end{proof}

Our next lemma depends on the Classification of the Finite Simple Groups
\cite{GLS, W}.

\begin{lem} \label{le:simplegroupsset}
Let $\Sigma$ be an infinite set of non-isomorphic finite non-abelian simple groups,
and let $K$ be any positive constant. Then there exists some $S_0 \in \Sigma$,
and some subgroups $M$ and $N$ of $S_0$ such that $M$ is solvable,
$N$ is a normal subgroup of $M$, and $M/N$ is Frobenius group with
Frobenius complement of order larger than $K$.
\end{lem}

\begin{proof}
Assume the lemma is false.
Suppose that $\Sigma$ and $K \ge 3$ provide a counterexample.
Let $t$ be a power of a prime satisfying $t-1>K$.
(This number $t$ will be fixed throughout the proof.)
The group $F$ of affine transformations
of the affine line over $\FF_t$ is a solvable
Frobenius subgroup of the group of permutations of
the affine line, so
$F$ is isomorphic to a subgroup of $S_t$. The Frobenius
complements of $F$ have order $|\FF_t^*|=t-1>K$. Since $S_t$ is isomorphic to a subgroup
of the alternating group $A_{t+2}$, in addition, $F$ is isomorphic to some subgroup of
the simple alternating group $A_{t+2}$.  It follows that no element of
$\Sigma$ contains any subgroup isomorphic to $A_{t+2}$.
In particular, if $A_n$ is isomorphic to some element of $\Sigma$
then $n \le t + 1$. It follows that
there are only a finite number of groups in $\Sigma$ which are isomorphic to
alternating groups.

Suppose that $S$ is a finite simple group of Lie type from table
\cite[Table I in p. 8]{GLS} and $S$ is isomorphic to some element of $\Sigma$.
The table characterizes $S$ by one of 16 \emph{types} and numbers $n$ and $q$.
If the number $n$ is not explicitly given by the table, we take it to be
the subindex, so, for example, if $S =\ ^3D_4(q)$ we set $n = 4$.
By Lemma \ref{le:frobeniussubgroupsofsimplegroups}, we have
$(q - 1)/2 \le K$, so that $q \le 2K +1$ is bounded above.
We next show that $n$ is also bounded above.\footnote{Our proof
uses implicitly the Weyl groups of the classical finite groups of Lie type,
but we avoid the explicit use of their theory for the benefit of
a potential reader unfamiliar with it.}
For this, we assume without loss of generality that $n \ge 9$.\footnote{This forces
$S$ to be isomorphic to some classical group.}

Before proving that $n$ is bounded, we pause to prove an auxiliary result.
Suppose that $G$ is a finite group, $G_0$ is a normal subgroup of $G$
such that $G/G_0$ is solvable, and $\phi : G_0 \to S$ is a surjective group
homomorphism whose kernel $\ker(\phi)$ is solvable, where $S$
is isomorphic to an element of $\Sigma$. We claim that $G$ cannot contain
subgroups $N$ and $T$ such that $T \trianglelefteq N$, $T$ is solvable and $N/T \simeq A_{t+2}$.
To prove the claim, assume that such groups $N$ and $T$ exist, and set
$N_0 = N \cap G_0$. Then the composition factors of $NG_0/G_0 \simeq N/N \cap G_0 = N/N_0$
are all solvable, so that $N_0$ contains exactly one non-solvable composition
factor and this factor is isomorphic to $A_{t+2}$. The same condition holds
for the image $\phi(N_0)$, since $\ker(\phi)$ is solvable.
Since $A_{t+1}$ contains a subgroup isomorphic to
a Frobenius group with complement larger than $K$, this contradicts our hypothesis.
Therefore the claim is proved.

Now suppose $S = A_n(q) \simeq \PSL_{n+1}(q)$. Then, by \cite[\S 3.3.3]{W}, there exist
groups $G = \GL_{n+1}(q)$, $G_0=\SL_{n+1}(q)$ and a homomorphism $\phi$ as in the previous paragraph.
Furthermore there exist subgroups $N$ and $T$ of $G$ such that $T \trianglelefteq N$, $T$ is solvable
and $N/T \simeq S_{n+1}$. If $n + 1 \ge t+2$ then this contradicts what we know from the
previous paragraph so that $n \le t$ in this case.
Suppose $S =\ ^2A_n(q) \simeq \PSU_{n+1}(q)$. Then, by \cite[\S 3.6.2]{W}, there exist
the groups $G = G_0 = \SU_{n+1}(q)$ and a homomorphism $\phi$ as in the previous paragraph.
Furthermore there exist subgroups $N$ and $T$ of $G$ such that $T \trianglelefteq N$, $T$ is solvable
and $N/T$ is isomorphic to the wreath product
$Z_2 \wr S_m$, where $m$ is the integer part of $(n+1)/2$.
If $m \ge t+2$ this contradicts what we know from the
previous paragraph, so that $n \le 2t+3$ in this case.
Suppose $S =B_n(q) \simeq \Omega_{2n+1}(q)$. Then, by \cite[\S 3.7.4]{W}, there exist
the groups $G = G_0 = \Omega_{2n+1}(q)$ and a homomorphism $\phi$ as in the previous paragraph.
Furthermore there exist subgroups $N$ and $T$ of $G$ such that $T \trianglelefteq N$, $T$ is solvable
and $N/T \simeq Z_2 \wr S_n$.
If $n \ge t+2$ this contradicts what we know from the
previous paragraph, so that $n \le t+1$ in this case.
Essentially the same argument shows that if
$S =D_n(q) \simeq \Pom_{2n}^+(q)$ then $n \le t+1$, and that if
$S =\ ^2D_n(q) \simeq \Pom_{2n}^-(q)$ then $n \le t+1$.
Suppose $S =C_n(q) \simeq \PSp_{2n}(q)$. Then, by \cite[\S 3.5.3]{W}, there exist
the groups $G = G_0 = \Sp_{2n}(q)$ and a homomorphism $\phi$ as in the previous paragraph.
Furthermore there exist subgroups $N$ and $T$ of $G$ such that $T \trianglelefteq N$, $T$ is solvable
and $N/T \simeq Z_2 \wr S_n$.
If $n \ge t+2$ this contradicts what we know from the
previous paragraph, so that $n \le t+1$ in this case.
Notice that, since $n \ge 9$, $S$ can not be any other finite simple
group of Lie type.
Hence, in all cases $n$ is bounded above, and there exist only a finite
number of simple groups of Lie type on the list which can be isomorphic to
groups in $\Sigma$.

Aside from alternating groups and groups of Lie type, \cite[Table II, p. 10]{GLS}
contains only the 26 sporadic simple groups and the Tits group $^2F_4(2)'$, and
the abelian simple groups.
Since the elements of $\Sigma$ are all non isomorphic to each other, they
are non-abelian, and they can only be isomorphic to a finite number of
alternating groups, and a finite number of finite simple groups of Lie type, we
conclude that $\Sigma$ is finite.
This contradicts our assumption,
and completes the proof of the lemma.
\end{proof}

\begin{lem} \label{le:frobeniusgroups}
Let $G$ be a Frobenius group of permutations of a finite set $\Omega$,
and let $G_{\omega}$ be one of its Frobenius complements.
Let $A$ be an abelian subgroup of $F$. Then
$[F : A] \ge |G_{\omega}|$.
\end{lem}

\begin{proof}
It follows directly from the definition of a Frobenius group that
$|\Omega| > 1$. Let $N^*=G\setminus\bigcup_{\omega\in\Omega}G_{\omega}$ and $N=\{1\}\cup N^*$.
A theorem of Frobenius \cite[(35.24)]{A} states that $N$ is a normal subgroup of
$G$. Fix some $\omega\in\Omega$. Since all
subgroups $\{G_{\omega'}\mid\omega'\in\Omega\}$ are conjugate, we have
$|G|=(|G_{\omega}|-1)|\Omega|+|N|$. Since $N$ acts freely on $\Omega$,
the map $N\ni\gamma\mapsto \gamma(\omega)\in\Omega$ is injective, so $|N|\leq|\Omega|$.
Applying the same argument to the action of $G_{\omega}$ on $\Omega\setminus\{\omega\}$
we get $|G_{\omega}|\leq|\Omega|-1$. We thus have
$$[G:G_{\omega}]=\frac{(|G_{\omega}|-1)|\Omega|+|N|}{|G_{\omega}|} \geq
\frac{(|G_{\omega}|-1)|\Omega|}{|G_{\omega}|}>|G_{\omega}|-1.$$
Since $[G:G_{\omega}]$ is an integer, this gives $[G:G_{\omega}]\geq |G_{\omega}|$.
Similarly
$$[G:N]=\frac{(|G_{\omega}|-1)|\Omega|+|N|}{|N|}=
\frac{(|G_{\omega}|-1)|\Omega|}{|N|}+1\geq |G_{\omega}|.$$
If $g\in G_{\omega}$ and $h\in G$ are nontrivial commuting elements, we have
$h\in G_{\omega}$, because the action of $g$ on $\Omega$ only fixes $\omega$.
Hence, if $A\subseteq G$ is abelian, then either $A\subseteq G_{\omega}$ for some
$\omega\in\Omega$ or $A\subseteq N$, so the lemma follows from the previous
bounds on $[G:G_{\omega}]$ and $[G:N]$.
\end{proof}

\section{Results} \label{sec:results}

For convenience, we now name the hypotheses which we will use throughout the rest of the
paper. (Recall from the introduction that if ${\mathcal C}$ is a set of finite groups
then $\mathcal T(\mathcal C)$ denotes the set of all $T \in \mathcal C$ such that
there exist primes $p$ and $q$, a Sylow $p$-subgroup $P$ of $T$,
and a normal Sylow $q$-subgroup $Q$ of $T$, such that $T = PQ$.)

\begin{hypA}
$M$ and $d$ are positive integers, $\mathcal C$ is a set of finite groups which is closed under taking subgroups,
and $\mathcal T = \mathcal T(\mathcal C)$
satisfies the Jordan property $\mathcal J(M,d)$.
\end{hypA}

\begin{lem} \label{le:largeprimes}
Assume Hypotheses A. Let $p$ be
any prime larger than $M$. Let $G \in \mathcal C$,
let $G_1 = \C_G(E(G))$, and let $P$ be a Sylow $p$-subgroup
$G_1$. Then $P$ is an abelian normal subgroup of $G_1$ and
$P$ can be generated by at most $d$ elements.
\end{lem}

\begin{proof}
Assume this is not the case, and that $P$ is a Sylow $p$-subgroup
which contradicts the statement of the lemma.
Since $P \in \mathcal T$, by hypotheses there exists an abelian
subgroup of $P$ of index at most $M < p$ and such that the subgroup
can be generated by at most $d$ elements, so that $P$ is
abelian and $P$ may be generated by at most $d$ elements.
Since we assume that $P$ is not simultaneously abelian and normal in $G_1$,
we deduce that $P$ is not a normal subgroup of $G_1$.
By Lemma \ref{le:nonnormalpsubgroup}, there exist a prime $q$ and
a nontrivial $q$-subgroup $Q$ of $G_1$ such that $q \ne p$ and $[P,Q] = Q$.
Furthermore, $P$ normalizes $Q$, so $PQ\subseteq G_1$ is a subgroup.
Since $PQ \in \mathcal T$, there exists some abelian subgroup
$J$ of $PQ$ such that $[PQ: J] \le M$.
Since $PQ$ is not abelian, $J$ is a proper subgroup of $PQ$.
By Lemma \ref{le:propersubgroup}, $[PQ : J] \ge p$.
This contradicts the fact that $p > M$. This contradiction
completes the proof of the lemma.
\end{proof}

\begin{lem} \label{le:centralizeroflargeprimes}
Assume Hypotheses A. Then there exists a positive integer $C_1$
with the following property.
Let $G \in \mathcal C$.
Set $G_1 = \C_G(E(G))$, and let $R$ be the product of
all the Sylow $p$-subgroups of $G_1$ for all primes
$p > M$. Let $G_2 = \C_{G_1}(R)$. Then $[G_1 : G_2] \le C_1$.
\end{lem}

\begin{proof}
Set $C_1= M^{(d + 1)M} $ so that $C_1\geq 1$. If $G_1 = G_2$ the results holds,
so we assume $G_1 \ne G_2$. In particular, $R$ is not trivial.
Let $q$ be any prime dividing $[G_1 : G_2]$, and let $Q$ be
a Sylow $q$-subgroup of $G_1$. By Lemma \ref{le:largeprimes},
$R$ is an abelian\footnote{$R$ is abelian because of the following
 general fact \cite[Theorem 1.26]{isa}:
 if $\Gamma$ is a finite group all of whose Sylow subgroups
 are normal then $\Gamma$ is the direct product of its Sylow subgroups.}
Hall subgroup of $G_1$, and it follows
that $q \le M$.
Fix $p_0$ to be some prime divisor of $|R|$, and
let $P_0$ be the Sylow $p_0$-subgroup of $R$.  Since
$QP_0 \in \mathcal T$, by hypotheses, there exists an
abelian subgroup $B$ of $QP_0$ such that $[QP_0 : B] \le M$,
and $B$ can be generated by at most $d$ elements.
Since $p_0 > M$, it follows that $B$ contains a Sylow
$p_0$-subgroup of $QP_0$, and therefore $B \supseteq P_0$.
We set $B_0 = Q \cap B$. Then
$[QP_0 : B] = [Q : B_0] \le M$.
Notice that $B_0$ is an abelian $q$-group generated
by at most $d$ elements, and it acts trivially on $P_0$.
Let $p$ be any prime divisor of $|R|$, and let $P$
be the Sylow $p$-subgroup of $R$. A similar argument
shows that $[B_0 : \C_{B_0}(P)] \le M$. Let
$B_1 = \C_{B_0}(R)$. Then
$B_1$ is the intersection of the $\C_{B_0}(P)$ as
we run through all the prime divisors $p$ of $|R|$.
Then $B_0/B_1$ is an abelian $q$-group of exponent at most
$M$ generated by at most $d$ elements, so that
$|B_0/B_1| \le M^d$. It follows that
$[Q : B_1] \le M^{d +1}$. Therefore,
the $q$-part of $[G_1 : G_2]$ is at most
$M^{d + 1}$. Therefore
$$
[G_1 : G_2] \le M^{(d + 1)M} = C_1,
$$
as desired.
\end{proof}

\begin{lem} \label{le:solvableinequality}
Assume Hypotheses A. Then there exists a positive constant $C_2$
such that whenever $G \in \mathcal C$ and $G_1 = \C_G(E(G))$,
then there exists an abelian subgroup
$A$ of $G_1$ such that $[G_1 : A] \le C_2$ and such that $A$ can be
generated by at most $d$ elements.
\end{lem}

\begin{proof}
Assume the notation of Lemma \ref{le:centralizeroflargeprimes}.
Set
$$C_2 = \left(\left(M!\right)^{M^M + d}\right)!\, M^M\, C_1.
$$
Let $G \in \mathcal C$.
Set $G_1 = \C_G(E(G))$, and let $R$ be the product of
all the Sylow $p$-subgroups of $G_1$ for all primes
$p > M$. Let $G_2 = \C_{G_1}(R)$. By Lemma \ref{le:largeprimes},
we know that $R$ is an abelian normal Hall subgroup of $G_2$
and $R$ can be generated by at most $d$ elements.
By the Schur-Zassenhauss Theorem, there exists a complement
$H$ to $R$ in $G_2$.
Then $H$ is a Hall
subgroup of $G_2$ for the set of primes smaller than or equal to $M$.
Since $H\subseteq C_{G_1}(R)$, the elements of $R$ commute with those
of $H$, and hence $G_2=H\times R$. This implies that $H$ is the only
Hall subgroup of $G_2$ for the set of primes not bigger than $M$. Hence, $H$
is characteristic in $G_2$, so $H$ is normal in $G$ (because $G_2\trianglelefteq G$: indeed, since by Lemma \ref{le:largeprimes}
$R\subseteq G_1$ is a normal Hall subgroup, $R$ is characteristic in $G_1$, so
$G_1\trianglelefteq G$ implies $G_2\trianglelefteq G$).
It then follows from the definition that $E(H) \subseteq E(G)$,
so that, since $E(G)$ centralizes $E(H) \subseteq G_1$, we know that $E(H)$ is abelian. This
implies, again from the definition, that $E(H) = 1$.
For each prime $p$ with $p \le M$, since $\OO_p(H) \in \mathcal T$, by hypotheses,
there exists an abelian subgroup $B_p$ of $\OO_p(H)$ such that $[\OO_p(H) : B_p] \le M$
and such that $B_p$ can be generated by at most $d$ elements.
Set $B$ to be the product of all the $B_p$ for $p \le M$.
Then $B$ is an abelian subgroup of $F(H)$ which can be generated
with at most $d$ elements, and $[F(H) : B] \le M^M$
($B$ is abelian because if $p\neq q$ then the elements of $O_p(H)$ commute
with those of $O_q(H)$).
It follows that $F(H)$ can be generated with at most $M^M + d$
elements.
Now $F(H)/\Phi(F(H))$ is abelian and of square-free exponent
by Lemma \ref{le:centralizeroffrattiniquotient}. Since
$F(H)/\Phi(F(H))$ can be generated by at most $M^M + d$ elements
and all prime divisors of $|F(H)/\Phi(F(H))|$ are smaller than or equal to $M$,
this implies that
$$
|F(H)/\Phi(F(H))| \le \left(M!\right)^{M^M + d}.
$$
Again by Lemma \ref{le:centralizeroffrattiniquotient}, this implies that
$$
[H : F(H)] \le \left(\left(M!\right)^{M^M + d}\right)!.
$$
Now set $A = B R$. Then $A$ is an abelian subgroup of $G_2$ such that
$$
[G_2 : A] = [H : B] \le \left(\left(M!\right)^{M^M + d}\right)!\, M^M
$$
and such that $A$ can be generated by at most $d$ elements. It follows that
$$
[G_1 : A] \le \left(\left(M!\right)^{M^M + d}\right)!\, M^M\, C_1 = C_2.
$$
Hence the lemma holds.
\end{proof}

\begin{cor} \label{co:solvable}
Assume Hypotheses A. Assume the notation of
Lemma \ref{le:solvableinequality}.
Let $S \in \mathcal C$ be solvable. Then, there
exists an abelian subgroup $A$ of $S$ such
that $[S : A] \le C_2$ and such that $A$ can be generated by
at most $d$ elements.
\end{cor}

\begin{proof}
This follows immediately from Lemma \ref{le:solvableinequality}
because if $G = S$ then $E(G) = 1$ and $G = G_1 = S$.
\end{proof}

\begin{lem} \label{le:quasisimple}
Assume Hypotheses A. Then there exists a positive constant $C_3$
such that whenever $G \in \mathcal C$ is quasisimple then $|G| \le C_3$.
\end{lem}

\begin{proof}
Let $\Sigma$ be a full set of representatives of the isomorphism
classes of simple groups $S$ of the form $S = G/\Z(G)$ for
some quasisimple $G \in \mathcal C$.

Suppose that $S \in \Sigma$ contains a solvable subgroup $H$ and a normal
subgroup $H_0$ of $H$ such that $H/H_0$ is a Frobenius group
with Frobenius complement of order $k$. Then there exists some
$J \in \mathcal C$ such that $J$ is solvable and $H/H_0$ is
a homomorphic image of $J$. By Corollary \ref{co:solvable}, there exists
some abelian subgroup of $J$ of index at most $C_2$ in $J$.
It follows that there exists an abelian subgroup of $H/H_0$ of
index at most $C_2$. By Lemma \ref{le:frobeniusgroups},
this tells us that $k \le C_2$.
By Lemma \ref{le:simplegroupsset}, this implies that
$\Sigma$ is a finite set.

Let $S \in \Sigma$. Since $S$ is a finite simple group,
there exists a finite
group $R$ called the Schur representation group\footnote{Also called
Schur covering group. See \cite[\S 5A]{isa} and \cite{kar}. The existence
of the Schur representation group is proved in \cite[Theorem 2.10.3]{kar}
(for notation see \cite[\S 2.7]{kar}). Note that the center $\Z(R)$ is isomorphic
to $H^2(S;\CC^{\times})$ and $R/\Z(R)\simeq S$. Since $S$ is finite, we also
have $H^2(S;\CC^{\times})\simeq H_2(S;\ZZ)$, see \cite[\S 2.7]{kar}.
In some respects, the Schur representation
group is an analogue for finite groups
of the universal covering group of a Lie group.} of $S$ such that
if $G$ is a quasisimple group and $G/\Z(G)$ is isomorphic to
$S$ then $G$ is a homomorphic image of $R$.
In particular, we have $|G| \le |R|$.
We can now set $C_3$ to be the maximum of the orders of the
Schur representation groups of the elements of $\Sigma$.
Hence the lemma holds.
\end{proof}

\begin{lem} \label{le:numberofsimples}
Assume Hypotheses A. Then there exists a positive constant $C_4$
such that whenever $G \in \mathcal C$ then the number of
composition factors of $E(G)/\Z(E(G))$ is at most $C_4$.
\end{lem}

\begin{proof}
We set $C_4 = d + M$.
Let $G \in \mathcal C$.
By, for example, \cite[Theorem 9.7]{isa}, $E(G)/\Z(E(G))$
is a direct product of simple groups. Say
$$
E(G)/\Z(E(G)) = S_1 \times \cdots \times S_n
$$
where $n$ is the composition length and
the $S_1,\ldots,S_n$ are non-abelian simple groups.
For each $i=1,\ldots,n$, pick a subgroup $T_i$ of order two of $S_i$
(we know that $T_i$ exists by the celebrated Feit--Thompson theorem
\cite{FT}, using the fact that $S_i$ is simple and non-abelian, hence
non-solvable),
and let $t_i$ be a $2$-element of $E(G)$ which projects onto
a generator for $T_i$. Each $t_i$ belongs to a different
component times the center of $E(G)$, so that they
all commute with each other, see for example
\cite[Theorem 9.4]{isa}. Let $B$ be the subgroup of $G$ generated
by the $t_1,\ldots,t_n$. Now $B$ is an abelian $2$-subgroup of $G$,
and it can not be generated by fewer than $n$ elements.
By our hypotheses, since $B \in \mathcal T$, there exists an abelian subgroup $C$
of $B$ such that $C$ can be generated by at most $d$ elements
and $[B : C] \le M$. Hence, $B$ can be generated by
at most $d + M = C_4$ elements. It follows that
$n \le C_4$, and the lemma holds.
\end{proof}

\begin{cor} \label{co:layer}
Assume Hypotheses A. Then there exists a positive constant $C_5$
such that whenever $G \in \mathcal C$ then
$|E(G)| \le C_5$.
\end{cor}

\begin{proof}
Set $C_5 = C_3^{C_4}$. The results follows immediately from
Lemma \ref{le:quasisimple} and Lemma \ref{le:numberofsimples}
since $E(G)$ is the product of as many components as
the composition length of $E(G)/\Z(E)$.
\end{proof}

\begin{teo} \label{te:main}
Assume Hypotheses A.
Then, there exists some positive integer $C_0$
such that $\mathcal C$ satisfies the Jordan property $\mathcal J(C_0,d)$.
\end{teo}

\begin{proof}
We use the notation of Lemma \ref{le:solvableinequality} and
Corollary \ref{co:layer}. Set $C_0 = C_5!\, C_2$.
Let $G \in \mathcal C$. Set $G_1 = \C_G(E(G))$. By
Lemma \ref{le:solvableinequality} there is an abelian subgroup
$A$ of $G_1$ such that $[G_1 : A] \le C_2$ and such that
$A$ can be generated by at most $d$ elements.
By Corollary \ref{co:layer}, we have $|E(G)| \le C_5$,
and this implies that $[G : G_1] \le C_5!$.
This implies that $[G : A] \le C_0$ and completes
the proof of the theorem.
\end{proof}

\end{document}